\documentclass[11pt]{amsart}
\usepackage{amsmath, amssymb}
\ifx\pdfoutput\undefined 
\usepackage{graphicx}
\else
\usepackage[pdftex]{graphicx}
\usepackage{epstopdf}
\fi
\usepackage{verbatim, manfnt, amscd}
\usepackage{tikz}

\newtheorem{theorem}{Theorem}[section]
\newtheorem{proposition}[theorem]{Proposition}
\newtheorem{lemma}[theorem]{Lemma}
\newtheorem{corollary}[theorem]{Corollary}

\newtheorem*{CGC}{Cheeger-Gromov $C^k$ Convergence Theorem}

\theoremstyle{definition}

\theoremstyle{remark}
\newtheorem{remark}[theorem]{Remark}

\numberwithin{equation}{section}

\begin{document}

\title{On Isospectral compactness in conformal class for 4-manifolds}

\author{Xianfu LIU}\address{School of Mathematical Sciences, University of Science and Technology of China, Hefei, Anhui, P. R. China}\email{lxfpqa@mail.ustc.edu.cn}
\author{Zuoqin WANG}\address{School of Mathematical Sciences, University of Science and Technology of China, Hefei, Anhui, P. R. China}\email{wangzuoq@ustc.edu.cn}

\thanks{The authors are supported in part by  NSF in China 11571331, NSF in China 11526212 and ``the Fundamental Research Funds for the Central Universities".}

\begin{abstract}
Let $(M, g_0)$ be a closed  4-manifold with positive Yamabe invariant and with $L^2$-small Weyl curvature tensor. Let $g_1 \in [g_0]$ be any metric in the conformal class of $g_0$ whose scalar curvature is $L^2$-close to a constant. We prove that the set of Riemannian metrics in the conformal class $[g_0]$ that are isospectral to $g_1$ is compact in the $C^\infty$ topology. 
\end{abstract}

\maketitle

\section{Introduction}

Let $M$ be a compact smooth manifold without boundary and let $g$ be a smooth Riemannian metric on $M$. We will  denote by $\Delta_g$ the Laplace-Beltrami operator associated to $g$. It is well known that the eigenvalues of $\Delta_g$ form a discrete sequence that tends to infinity:
\[
\mathrm{Spec}(\Delta_g): 0 =\lambda_0 < \lambda_1 \le \lambda_2 \le \lambda_3 \le \cdots \quad \to \infty.
\] 
Although one can't compute individual eigenvalues explicitly in most cases, it has long been known that the sequence $\mathrm{Spec}(\Delta_g)$ is quite rigid, at least in the $k \to \infty$ limit. For example, the Weyl's law states that the leading asymptotic of $\lambda_k$'s is completely determined by the dimension and the volume of $(M, g)$. A very interesting question is to study the \emph{exact} relation between the geometry of $(M, g)$ and $\mathrm{Spec}(\Delta_g)$. This turns out to be a very subtle question: lots of theorems (both in positive and negative directions) have been proved, while many major  conjectures are still widely open. For some of the conjectures and their current status, we refer to \cite{Yau}, \cite{Zel}  and references therein.   
 
Two Riemannian metrics $g$ and $g'$ on $M$ are said to be isospectral if $\mathrm{Spec}(\Delta_g)=\mathrm{Spec}(\Delta_{g'})$. People had found plenty of examples of isospectral pairs (e.g. \cite{Mil}, \cite{Sun}), or even families (e.g. \cite{Gor}, \cite{BrG}), of Riemannian metrics. However, it is still believed that the set of isospectral metrics on any smooth manifold is not too large. The famous isospectral compactness problem asks: Given any compact smooth manifold $M$, for any Riemannian metric $g$, is the set of Riemannian metrics on $M$ that is isospectral to $g$ compact in the $C^\infty$ topology? In other words, does any sequence of isospectral metrics admits a convergent subsequence in the $C^\infty$ topology?

One of the first remarkable works on the isospectral compactness problem was done by  B. Osgood, R. Phillips and P. Sarnak \cite{OPS}: they proved the compactness of isospectral metrics on any given compact Riemann surface. For manifolds of dimension greater than two, it is still not known whether the isospectral sets of metrics on a given manifold are compact or not. However, if one restricted to the isospectral metrics in the same conformal class, then it was proved by A. Chang and P. Yang (\cite{CY1},  \cite{CY2}) that for three dimensional compact manifolds, the isospectral metrics in the same conformal class is compact. People also studied isospectral compactness  under other extra assumptions, see e.g. \cite{And}, \cite{BPY},\cite{BPP}, \cite{Gur} and \cite{Zhou}. We remark that even inside the same conformal class, one can find isospectral families of non-isometric Riemannian metrics (\cite{BrG}). 

Before we continue to describe the isospectral compactness results for 4 dimensional manifolds, we would like to say a few words about the ideas of \cite{CY1} for 3-dimensional manifolds. Recall that any Riemannian metric in the conformal class $[g_0]$ of $g_0$ is of the form $g=u^2 g_0$. So to prove the compactness of isospectral Riemannian metrics, one need to prove  the isospectral compactness of the corresponding conformal factors. In other words, suppose $u_j \in C^\infty(M)$ be a sequence of conformal factors so that $g_j=u_j^2g_0$ are isospectral, one need to prove that the sequence of functions $\{u_j\}$ admits a convergent subsequence. [This is not quite precise, since one has to modulo the effect of isometries. See the remark after theorem \ref{mainthm} below.] In their proof of the isospectral compactness in conformal class for 3-manifolds, A. Chang and P. Yang introduced the following ``non-blowup" condition for the sequence of the conformal factors $u_j$'s: 
\begin{equation}\label{star}
\aligned
\mbox{\ There exist positive constants \ }  r_0, l_0 & \mbox{\ so that for all \ }j, \qquad \\
 \mathrm{Vol}\{x   |  u_j(x) \ge r_0\}  \ge l_0   \mathrm{Vol}&(M, g_0). 
\endaligned
\end{equation} 
Their proof are then divided into two parts, the easier part being proving the isospectral compactness for sequences  $u_j$'s satisfying (\ref{star}), 
 while the harder part is to verify that the condition (\ref{star}) is always true under the isospectral assumption: They showed that if  (\ref{star}) fails, then  $(M, g_0)$ is conformal to the standard $(S^3, g_0)$.

The isospectral compactness in conformal class problem for 4 dimensional manifolds was analyzed in \cite{Xu1}, \cite{Xu2} and \cite{ChX}. For example, in \cite{ChX} they proved the isospectral compactness in conformal class under the extra conditions 
\[
\int_M S_{g_1} dv_{g_1}<\frac 6{C_s} \left(\int_M dv_{g_1}\right)^{1/2}
\]
and
\[
\int_MS_{g_1}^2dv_{g_1}-\frac{(\int_MS_{g_1}dv_{g_1})^2}{\int_Mdv_{g_1}}\leq\frac{11^2}{52^2C_s^2},  
\]  
where $C_s$ is the Sobolev constant in (\ref{SI}) below. Then showed that   the first inequality implies that the conformal factors $u_j$'s satisfy the condition (\ref{star}). Moreover, from these two inequality they proved that  $\int |R_g|^4dv$ is bounded, which implies that the conformal factors are bounded from below and above uniformly. Note that as Chen and Xu remarked, their arguments also works for positive scalar curvature case.

One of the main tools in studying isospectral compactness problem is the Sobolev inequality of Aubin (c.f. \cite{Aub}). When restricted to 4-dimensional Riemannian manifold $(M^4,g_0)$, the inequality takes the form
\begin{equation}\label{SI}
\left(\int_{M^4}|f|^4dv_0\right)^{{1}/{2}}\leq C_s\int_{M^4}|\nabla f|^2dv_0+K_s\int_{M^4}f^2dv_0
\end{equation}
for any function $f\in H_1^2(M^4)$, where $C_s$ and $K_s$ are positive constants.
In what follows we will always make the following assumption:
\[
\text{\bf Assumption: } g_0 \text{ is a metric with constant scalar curvature }  S_0>0. 
\] 
Let $C_s$ and $K_s$ be the Sobolev constants as in (\ref{SI}) for the metric $g_0$. We define a constant (which depends only on $g_0$) 
\begin{equation}\label{C0}
C_0 :=\max(C_s, \frac{6K_s}{S_0}). 
\end{equation}
Note that according to \cite{AuL} (see also \cite{Bec}), one has $C_0 \ge \frac{\sqrt 6}{8 \pi}$. 

In this paper, we will study isospectral compactness in conformal class for 4-manifolds with positive Yamabe invariant. Instead of bound the scalar curvature (as in the first condition of \cite{ChX} cited above), we will assume that the Weyl curvature tensor has a small $L^2$-norm. Since the norm of the Weyl curvature tensor is a conformal invariant,   we only have one restriction (see the condition (\ref{SCB}) below) on the scalar curvature. Note that the condition (\ref{SCB}) is in fact a condition on the spectrum.  

Our main theorem in this paper is 
\begin{theorem}\label{mainthm}
Let $(M,g_0)$ be a compact Riemannian 4-manifold with positive Yamabe invariant. Suppose the Weyl curvature of $(M, g_0)$ satisfies 
\begin{equation}\label{WB}
\int|W|_{g_0}^2dv_{0}\leq\frac{1}{625C_0^2} ,
\end{equation}
where $C_0$ is as in (\ref{C0}).
Suppose $g_1\in [g_0]$ be any metric in the same conformal class of $g_0$ satisfying 
\begin{equation}\label{SCB}
\int_MS_{g_1}^2dv_{g_1}-\frac{(\int_MS_{g_1}dv_{g_1})^2}{\int_Mdv_{g_1}}\leq\frac{1}{64C_0^2}.    
\end{equation}
Then the set of Riemannian metrics $g$ in $[g_0]$ which are isospectral to $g_1$ is compact in the $C^{\infty}$-topology.  
\end{theorem}
\begin{remark}
For the case $M=S^4$ with $g_0=g_{can}$ the canonical round metric, the compactness is in the following sense: For any sequence $g_j$ in $[g_0]$ that are isospectral to each other, there is a choice of conformal factors $\{u_j\}$'s so that each $g_j$ is isometric to $u_j^2 g_0$, and $\{u_j\}$'s has a convergent subsequence.  
\end{remark}

We shall say a few words of the proof.  As in \cite{CY1}, we will  first prove   theorem \ref{mainthm}  under the extra assumption (\ref{star}). The argument is closely related to that of \cite{Xu1} and \cite{ChX}, i.e. we first deduce that such $u_j$'s are uniformly bounded both from below and from above.  As noticed by \cite{CY2} and \cite{Xu2},  modulo isometries  the conformal factors on the standard $S^4$ can be chosen to be satisfying (\ref{star}).  So in particular  this implies that theorem \ref{mainthm} is true for $S^4$. This will be done in section 4. For the rest of this paper, we will show that (\ref{star}) holds under the condition of theorem \ref{mainthm}. Motivated by \cite{CY1}, we will prove  the following ``conformal sphere theorem": 

\begin{theorem}\label{mainthm2}
Let $(M, g_0)$ be a 4-dimensional closed Riemannian manifold with positive Yamabe invariant and satisfies
\begin{equation}
\int_M |W|^2 dv_0 < 16\pi^2 . 
\end{equation}  
Let $\{u_j\}$ be a sequence of positive smooth functions $M$ so that 
\begin{enumerate}
\item The integral $\int_M|R|^4dv_{g}$ is bounded for the sequence $g_j=u_j^2g_0$, 
\item There exist $x_0 \in M$ and a sequence of constants $C_j>0$ with $C_j\to\infty$ so that the sequence
$\{C_j u_j\}$
 converges uniformly on compact subset of $M\setminus\{x_0\}$ to the Green's function of the conformal Laplacian $L$. 
\end{enumerate}
Then $(M, g_0)$ is conformally equivalent to $(S^4, g_{can})$. 
\end{theorem}

The main ingredients in proving theorem \ref{mainthm2} are the conformal sphere theorem of \cite{CGY}, and the classification of complete connected flat manifolds in \cite{Wolf}. We remark that the conformal sphere theroem of \cite{CGY} assumes 
\[
\int_M |W|^2 dv_0 < 16 \pi^2 \chi(M),
\]
which requires $\chi(M)>0$. There are plenty of closed manifolds with positive Yamabe invariant and non-positive Euler characteristic. For a generalized sphere theorem, c.f. \cite{ChZ}. We will prove theorem \ref{mainthm2} in the second half of section 6. In section 5 and the first half of section 6 we will show that under the conditions of theorem \ref{mainthm}, if the non-blowup condition (\ref{star}) fails, then the condition in theorem \ref{mainthm2} must hold. So the proof of the main theorem is completed.    



\section{Preliminaries}
 
\subsection{Heat invariants}

One of the main tools used in studying the isospectral compactness problem is the heat trace expansion. It is well known that as $t \to 0$, 
\[
\mathrm{Tr}(e^{-t\Delta}) = \sum_i e^{-t\lambda_i} \sim {(4\pi t)^{-\frac{n}{2}}}(a_0+a_1t+a_2t^2+a_3t^3+\cdots),
\]
where $a_0, a_1, a_2, \cdots$ are integrals of derivatives of curvature terms on $M$. For Riemannian manifolds of dimension 4, the first several heat invariants are explicitly given by 
\begin{subequations}\label{heat}
\begin{align}
a_0 &= \int_Mdv_g=\mathrm{Vol}(M), \label{heata0}\\
a_1 &= \frac{1}{6}\int_MS_gdv_g, \label{heata1} \\
a_2 &= \frac{1}{180}\int_M \bigg(|W|^2+|B|^2+\frac{29}{12}S_g^2\bigg)dv_g, \label{heata2}
\end{align}
\end{subequations} 
and (c.f. \cite{Xu2}, \cite{Sak}, \cite{Gi1}, \cite{Gi2})
\begin{equation}\label{a3}
\aligned
a_3=\frac{1}{7!}\int_M\bigg(&-\frac{7}{3}|\nabla W|^2-\frac{4}{3}|\nabla B|^2-\frac{152}{9}|\nabla S_g|^2\\
&+ 5S_g|W|^2+\frac{50}{9}S_g|B|^2+\frac{185}{54}S_g^3\\
&+ \frac{38}{9}W^{ijkl}W_{kl}^{\ \ mn}W_{mnij}+12B^{ij}W_{i}^{\ klm}W_{jklm}\\
&+ \frac{40}{3}B^{ij}B^{kl}W_{ikjl}+\frac{56}{9}W^{ijkl}W_{i\ k}^{m\ n}W_{jmln}\bigg)dv_g,
\endaligned
\end{equation}
where $S_g, Ric, B, W, R$ and $dv_g$ are the scalar curvature, the Ricci curvature tensor, the traceless Ricci curvature tensor, the Weyl curvature tensor, the full Riemannian curvature tensor, and the volume element associated to the given metric $g$ respectively. For 4-manifolds, $R$, $W$, $B$ and $S$ are related by 
\[
R_g=W_g+\frac 12 B_g \textcircled{$\wedge$} g + \frac S{24}  g  \textcircled{$\wedge$} g, 
\]
where $\textcircled{$\wedge$}$ is the Kulkarni-Nomizu product. 
In particular, 
\begin{equation}\label{RWBS}
|R_g|^2 = |W_g|^2+2 |B_g|^2 + \frac 16  S_g^2.
\end{equation}

\subsection{Conformal change of metric}

Let $M$ be a 4-manifold. Then under the conformal change $g=u^2g_0$,  the volume forms of the metrics $g$ and $g_0$ are related by
\[
dv_g = u^4 dv_0,
\]
while the corresponding scalar curvatures  are related by the equation 
\begin{equation}\label{ConfSc}
6\Delta_{g_0} u+S_gu^3=S_0u.
\end{equation}
Another very important fact for us is that the integral 
\[\int_M |W_g|^2dv_g\] 
is invariant under the conformal change.

We also notice that on 4-manifolds, the quantity 
\[
\int_M S_g^2 dv_g
\]
is a spectral invariant if we assume that the metrics sit  in the same conformal class. For a proof, c.f. \cite{ChX}.

\section{Some norm estimates}

We recall that $|W|^2=W^{ijkl}W_{ijkl}$. 

\begin{lemma} \label{lem3.2}
Under the assumption (\ref{WB}), we have
\begin{subequations}\label{W3R2WRW2}
\begin{align}
& \left|\int_M W^{ijkl}W_{kl}^{\ \ mn}W_{mnij}dv_g \right| \leq \frac{1}{25C_0}\left(\int_M|W_g|_g^4dv_g\right)^{{1}/{2}}, \label{W3R2WRW2a}\\
& \left|\int_M B^{ij}B^{kl}W_{ikjl} dv_g\right|\leq\frac{1}{25C_0}\left(\int_M|B_g|_g^4dv_g\right)^{{1}/{2}}, \label{W3R2WRW2b} 
\end{align}
\end{subequations}
and
\begin{equation*}
\tag{3.1c} 
\left|\!\int_M\!\!B^{ij}W_{i}^{\ klm}W_{jklm}dv_g\right|\!\leq\!\frac{1}{50C_0}\!\!\left[\eta\!\left(\!\int_M\!|B_g|_g^4dv_g\!\right)^{{1}/{2}}\!
+\!\!\frac{1}{\eta}\!\left(\!\int_M\!|W_g|_g^4dv_g\!\right)^{{1}/{2}}\!\right],\label{W3R2WRW2c}
\end{equation*}
where $\eta$ is any positive constant.
\end{lemma}
\begin{proof}
(\ref{W3R2WRW2a}) follows from 
\[
\int_M|W_g|_g^3dv_g  \leq  \left(\int_M|W_g|_g^2dv_g \right)^{{1}/{2}} \left(\int_M|W_g|_g^4dv_g \right)^{{1}/{2}} \leq \frac{1}{25C_0}\left(\int_M|W_g|_g^4dv_g \right)^{{1}/{2}}.
\]
The proof of (\ref{W3R2WRW2b}) is similar.  To prove (\ref{W3R2WRW2c}), one only need to notice 
\[\aligned
 \int_M|B_g|_g|W_g|_g^2dv_g & \leq \left(\int_M|B_g|_g^2|W_g|_g^2dv_g\right)^{{1}/{2}} \left(\int_M|W_g|^2dv_g\right)^{{1}/{2}} 
 \\& \leq\frac{1}{25C_0}\left(\int_M|B_g|_g^4dv_g\right)^{{1}/{4}}\left(\int_M|W_g|_g^4dv_g\right)^{{1}/{4}}
\endaligned\]
and use the fact that for any positive $a, b$ and $\eta$, 
$ab \le \frac 12 (\eta a^2 + \frac 1\eta b^2).$
\end{proof}

\begin{lemma} \label{lem3.3} 
Assume (\ref{C0}), then one has
\begin{subequations}\label{S4W4R4}
\begin{align}
& \left(\int_MS_g^4dv_g\right)^{{1}/{2}}\leq C_0\int_M|\nabla S_g|_g^2dv_g+\frac{C_0}{6}\int_MS_g^3dv_g, \label{S4W4R4a}\\
& \left(\int|W|_g^4dv_g \right)^{{1}/{2}}\leq C_0\int|\nabla W|_g^2dv_g+\frac{C_0}{6}\int S_g|W|_g^2dv_g, \label{S4W4R4b} \\
& \left(\int_M|B|_g^4dv_g\right)^{{1}/{2}}\leq C_0\int_M|\nabla B|_g^2dv_g+\frac{C_0}{6}\int_MS_g|B|_g^2dv_g. \label{S4W4R4c}
\end{align}
\end{subequations}
\end{lemma}
\begin{proof}
The proofs of (\ref{S4W4R4a}) and (\ref{S4W4R4c}) are essentially the same as in \cite{ChX}. In fact, according to the Sobolev inequality (\ref{SI}) and the definition of $C_0$, 
\[\left(\int_MS_g^4dv_g\right)^{{1}/{2}}  = \left(\int_M S_g^4u^4 dv_0\right)^{{1}/{2}} \leq C_0\int_M|\nabla(S_gu)|_0^2dv_0+K_s\int_MS_g^2u^2dv_0.\]
For the first term, we have (c.f. the proof of lemma 4.2 of \cite{ChX})
\[\aligned
\int_M|\nabla(S_gu)|_0^2dv_0 & = \int_M|\nabla_0 S_g|_0^2u^2dv_0+2 \int_M \langle \nabla S_g,\nabla u \rangle_0S_gudv_0+\int_MS_g^2|\nabla u|_0^2dv_0 \\
& = \int_M|\nabla_g S_g|_g^2dv_g+ \int_MS_g^2(\frac{S_gu^3-S_0u}{6})udv_0.
\endaligned\]
So we get
\[
\left(\int_MS_g^4dv_g\right)^{{1}/{2}}  \le  C_0\int_M|\nabla S_g|_g^2dv_g+\frac{C_0}{6}\int_MS_g^3dv_g+(K_s-\frac{C_0S_0}{6})\int_MS_g^2u^2dv_0,\]
which proves (\ref{S4W4R4a}). To prove (\ref{S4W4R4b}),
we start with  
\[
\left(\int|W|_g^4dv_g \right)^{{1}/{2}} = \left(\int_M |W|_g^4u ^4 \right)^{{1}/{2}} dv_0 
\leq  C_0\int_M|\nabla|W|_gu|^2_0dv_0+K_s\int_M|W|_g^2u^2dv_0,
\]
and for the first term, we use
\[\aligned
\int_M|\nabla|W|_gu|^2_0dv_0 &= \int_M |\nabla|W|_g|_0^2 u^2 dv_0+\frac{1}{2}\int_M \langle \nabla|W|_g^2,\nabla u^2\rangle_0dv_0 + 
\int_M|W|_g^2|\nabla u|_0^2dv_0 \\
& = \int_M|\nabla|W|_g|_g^2dv_g-\frac{1}{2}\int_M|W|_g^2\Delta_0u^2dv_0+ \int_M|W|_g^2|\nabla u|_0^2dv_0 
\\& = \int_M|\nabla|W|_g|_g^2dv_g - \int_M|W|_g^2u\Delta_0udv_0
\\& = \int_M|\nabla|W|_g|_g^2dv_g+ \int_M|W|_g^2u\frac{S_gu^3-S_0u}{6}dv_0
\\ & \le \int_M|\nabla W|_g^2dv_g+\frac{1}{6}\int_M|W|_g^2S_gdv_g -\frac{S_0}{6} \int_M|W|_g^2u^2dv_0.
\endaligned\]
The proof of (\ref{S4W4R4c}) is similar.
\end{proof}
 
\begin{lemma}\label{lem3.4} 
Assume (\ref{SCB}), then 
\begin{subequations}\label{SW2S3SR2}
\begin{align}
& \left|\int_MS_g|W|_g^2dv_g\right|\leq \frac 1{8C_0} \left(\int_M|W|_g^4u^4dv_0\right)^{{1}/{2}}+1080\frac{a_1a_2}{a_0}, \label{SW2S3SR2a}\\
& \left|\int_MS_g^3dv_g\right|\leq \frac 1{8C_0}\left(\int_MS_g^4dv_g\right)^{{1}/{2}}+\frac{12960}{29}\frac{a_1a_2}{a_0}, \label{SW2S3SR2b} \\
& \left|\int_MS_g|B|_g^2dv_g\right|^2\leq \frac 1{8C_0}\left(\int_M|B|_g^4dv_g\right)^{{1}/{2}}+1080\frac{a_1a_2}{a_0}. \label{SW2S3SR2c}
\end{align}
\end{subequations}
\end{lemma}
\begin{proof}
The estimate (\ref{SW2S3SR2a}) follows from
\[\begin{split}
\int_MS_g|W|_g^2dv_g &= \int_M \left(S_gu^2-\frac{\int_MS_gu^4dv_0}{\int_Mu^4dv_0}u^2\right)|W|_g^2u^2dv_0 \\
& \qquad\qquad +\frac{\int_MS_gu^4dv_0}{\int_Mu^4dv_0}\int_M|W|_g^2u^4dv_0\\
&\leq \bigg[\!\int_M\!S_g^2u^4dv_0-\frac{(\int_MS_gu^4dv_0)^2}{\int_Mu^4dv_0}\bigg]^\frac{1}{2}\!\left(\int_M\!|W|_g^4u^4dv_0\right)^{{1}/{2}}+1080\frac{a_1a_2}{a_0},
\end{split}\]
and the proofs of (\ref{SW2S3SR2b}) and (\ref{SW2S3SR2c}) are similar.
\end{proof}

Substituting the estimates (\ref{W3R2WRW2a})-(\ref{SW2S3SR2c}) into the heat invariant $a_3$,  we get  
\begin{lemma}\label{Lemma3.8} 
Under the assumptions of lemmas \ref{lem3.2}, \ref{lem3.3} and \ref{lem3.4}, we have
\begin{equation}\label{WBS4bound}
\aligned
\frac{7}{3}\left(\int|W|_g^4dv_g\right)^{{1}/{2}}& +\frac{4}{3}\left(\int|B|^4_gdv_g\right)^{{1}/{2}}+\frac{152}{9}\left(\int_MS_g^4dv_g\right)^{{1}/{2}}\\
\leq & \left[\frac{97}{18} \frac 18+\frac{1}{25}(\frac{94}{9}+\frac{6}{\eta})\right]\left(\int_M|W_g|_g^4dv_g\right)^{{1}/{2}}\\
& +\left[\frac{1}{25}(6\eta+\frac{40}{3}) +\frac{52}{9} \frac 18 \right]\left(\int_M|B|_g^4dv_g\right)^{{1}/{2}}
\\& +\frac{337}{54} \frac 18 \left(\int_MS_g^4dv_g\right)^{{1}/{2}}+{18800C_0} \frac{a_1a_2}{a_0} -7!a_3C_0.
\endaligned\end{equation}
\end{lemma}

As a consequence, we can prove
\begin{proposition}\label{NER}
Let $g$ be any metric as described in theorem \ref{mainthm}, then there exist constants $A_1, A_2$ such that
\begin{equation}\label{nablaRB}
\int_M|\nabla R|_g^2dv_g\leq A_1
\end{equation}
and
\begin{equation}\label{R4B}
\int_M|R|^4dv_g\leq A_2.
\end{equation}
\end{proposition}

\begin{proof}
Take $\eta=\frac{1}{5}$ in (\ref{WBS4bound}). It is easy to see that the quantity   
\[\left(\int_MS_g^4dv_g\right)^{{1}/{2}}+\left(\int_M|B|_g^4dv_g\right)^{{1}/{2}}+\left(\int_M|W|_g^4\right)^{{1}/{2}}\] 
is bounded. In view of (\ref{RWBS}) we get 
\[\int_M|R|_g^4dv_g\leq A_2.\]

To obtain a bound on $\int_M|\nabla R|_g^2dv_g$, we notice that according to the formula (\ref{a3}) of $a_3$, we can write
\[\int_{M}\left[\frac{7}{3}|\nabla W|^2+\frac{4}{3}|\nabla B|+\frac{152}{9}|\nabla S_g|^2\right]dv_g\]
as
\[\aligned
\int_M & \left[  5S_g|W|^2+\frac{50}{9}S_g|B|^2+\frac{185}{54}S_g^3 +\frac{38}{9}W^{ijkl}W_{ij}^{~~mn}W_{klmn}\right. \\
& \left. +12B^{ij}W_{i}^{~klm}W_{jklm} +\frac{40}{3}B^{ij}B^{kl}W_{ijkl}+\frac{56}{9}W^{ijkl}W_{i~k}^{~m~n}W_{jmln}\right]dv_g-7!a_3,
\endaligned\]
which can be controlled by the integral 
\[\int_M|R|^4_gdv_g\]
using the H\"older inequality and the estimates above. 
\end{proof}


\section{The proof of  theorem \ref{mainthm} under condition (\ref{star})}

We will start with proving the following proposition, which claims that under the condition (\ref{star}), the conformal factors $u_j$'s are uniformly bounded. We remark that for the negative scalar curvature case, this was proved in \cite{Xu1}.
\begin{proposition} \label{upperlower}
Let $g_0$ be a metric on $M^4$, and $g=u^2g_0$  a metric satisfy conditions (\ref{star}), (\ref{WB}) and (\ref{SCB}), then there exist constants $C_{\alpha}, C_\beta>0$ such that   $C_{\alpha} \le u \le C_\beta$.
\end{proposition}
\begin{proof} 
We first notice that although the proposition 3.1  in \cite{ChX} was stated under the condition $S_0<0$, the same argument works for the case $S_0>0$ without any change.  So there exists a constant $C_1$ such that
\begin{equation}
\int_Mu^{-4}dv_0 \le  C_1. 
\end{equation} 
In particular, this also implies $\int_Mu^{-1}dv_0 \le  C_1'$ for some constant $C_1'$.

Also in the proof of proposition \ref{NER} above we see that there is a constant $C_2$ so that 
\[
\int_MS_g^4u^4 dv_0 = \int_MS_g^4 dv_g \le C_2.
\] 
So by (\ref{ConfSc}), one can find a constant $C$ so that 
\begin{equation*}
\aligned
\int_M(\Delta_{g_0} u)^4u^{-8}dv_0 &= \frac{1}{6^4}\int_M|S_gu^3-S_0u|^4u^{-8}dv_0\\
&\leq\frac{16}{6^4}\int_M \left(|S_gu^3|^4+|S_0u|^4 \right)u^{-8}dv_0\\
&\le C. 
\endaligned\end{equation*}  

On the other hand, if we let $G$ be the Green's function on $M$ (with respect to $g_0$) which can be assumed to be positive everywhere, then $G_x(y)= G(x, y)$ is $L^p$ integrable for $p<2$ since $M$ is 4-dimensional. Since
\[\Delta \frac 1u =-\frac 1{u^{2}}\Delta u+\frac 2{u^{3}}|\nabla u|^2,\]
we get from Green's formula that for any point $x\in M$,
\[\aligned
\frac 1{u(x)}-(\int_Mdv_0)^{-1}\int_M\frac 1u dv_0& =-\int_M G(x,y)\left[-\frac 1{u^{2}}\Delta u +\frac 2{u^{3}}|\nabla u|^2\right]dv_0(y) \\
& \le  \int_M G(x,y)\left[ \frac 1{u^{2}}\Delta u \right]dv_0(y) \\
& \le \|G_x\|_{L^{4/3}}  \left \|\frac 1{u^{2}}\Delta u \right \|_{L^{4}}.
\endaligned\]
It follows that there exists a constant $C_{\alpha}>0$ which does not depend on $u$ such that for any $x \in M$, 
\[u(x)\geq C_\alpha>0.\]

The proof of the upper bound is similar to that of proposition 1 in \cite{Xu1}. So we will omit the details here. 
\end{proof}

We shall use the $C^k$ version of the Cheeger-Gromov compactness :
\begin{CGC}[\cite{OPS}, \cite{ChX},  \cite{Ch}, \cite{Gro}]
For any $k$, the space of $n$-dimensional Riemannian manifolds satisfying the bounds
\begin{subequations}\label{CkCpt}
\begin{align}
& |\nabla^jR|_{C_0}\leq \Lambda(j), \qquad  j\leq k, \label{CkCpta}\\
&  \mathrm{Vol}(M,g)\geq v>0, \label{CkCptb} \\
&  \mathrm{diam}(M, g)\leq D \label{CkCptc}
\end{align}
\end{subequations}
is (pre)compact in the $C^{k+1,\alpha}$ topology on $M$.
\end{CGC}

More precisely, given any $\alpha<1$, any sequence of metrics $\{g_i\}$ on $M$ satisfying  (\ref{CkCpta})-(\ref{CkCptc})  has a subsequence converging in the $C^{k+1,\alpha'}$ topology for $\alpha'<\alpha$ to a limit $C^{k+1,\alpha}$ Riemannian metric $g$ on $M$. 

\begin{proof}[Proof of theorem \ref{mainthm} under the condition (\ref{star})]
One need to verify (\ref{CkCpta})-(\ref{CkCptc})  for Riemannian metrics $g_j=u_j^2g_0$ satisfying the condition (\ref{star}).
 
The bound (\ref{CkCpta}) follows from proposition \ref{NER}, proposition \ref{upperlower} above, and  proposition 3 of \cite{Xu1}. 
The bound  (\ref{CkCptb}) follows from the first heat invariant $a_0$. 
To prove (\ref{CkCptc}), one need to use the fact that $C_\alpha \le u_j \le C_\beta$. It follows that $C_\alpha^2 g_0 \le g_j= u_j^2 g_0 \le C_\beta^2 g_0$.  
Now for any $p, q \in M$, let $\gamma$ be the minimal geodesic (with respect to the metric $g_1$) connecting $p$ and $q$. Then 
\[
\mathrm{dist}_{g_j}(p, q) \le L_{g_j}(\gamma) \le C_\beta^2 L_{g_1}(\gamma)=C_\beta^2 \mathrm{dist}_{g_1}(p,q) \le C_\beta^2 \mathrm{diam}(M, g_1). 
\]
So (\ref{CkCptc}) follows. 
\end{proof}

 As noticed by \cite{CY1} and \cite{Xu2},  on $(S^4,g_{can})$ where $g_{can}$ is the canonical round metric on $S^4$, if $g_j=u_j^2g_{can}$ is a sequence of conformal metrics satisfying 
\[C_0=\int_{S^4}u_j^4dv_0\]
and $\lambda_1(g_j)\geq\Lambda>0$, then there exist a sequence of conformal factors  $v_j$'s such that each $v_j^2g_{can}$ is  isometric  to $u_j^2g_{can}$, and $v_j$'s satisfy the condition (\ref{star}) {with universal $r_0$, $l_0$ depending only on $C_0$ and $\Lambda$}. As a consequence, 
\begin{corollary}[\cite{Xu2}]\label{S4St}
Theorem \ref{mainthm} holds for $S^4$ with the canonical round metric $g_{can}$. 
\end{corollary}

\section{Mass Concentration}

For the rest of the paper, we will study the isospectral compactness for the sequence $\{g_j=u_j^2g_0\}$ under the assumption that the condition $(\ref{star})$ fails for any subsequence of $\{u_j\}$.  
We will show this can happen only when some subsequence of $u_j$ has its mass ``concentrate" at some point $x_0\in M$.

We will start with a technical lemma that we will need several times later. For simplicity we denote
\begin{equation}
C_g =\int_MS_{g}^2dv_{g}-\frac{(\int_MS_{g}dv_{g})^2}{\int_Mdv_{g}}.
\end{equation}
We notice that for $g$ in the same conformal class,  $C_g$ is in fact a spectral invariant.  
\begin{lemma}
\label{Lem4.1}  
Let $(M,g_0)$ be a 4-dimension manifold, $g=u^2g_0$, and $\eta$ a positive cut-off function which will be chosen later. Then for $\beta\neq0$ and $\beta\neq-1$ we have
\begin{equation}
\aligned
\left(\int_M{\omega^4\eta^4dv_0}\right)^{{1}/{2}} &\leq 2C_s\left(6\frac{A_\beta}{|\beta|}+1\right)\int_M|\nabla\eta|^2\omega^2dv_0+
(A_\beta |S_0|+K_s) \int_M\omega^2\eta^2dv_0  \\
 &+  C_sC_g^{{1}/{2}}A_\beta\left(\int_M\omega^4\eta^4dv_0\right)^{{1}/{2}}+A_\beta C_s \bigg|\frac{\int_MS_gu^4dv_0}{\int_Mu^4dv_0}\bigg| \int_Mu^2\omega^2\eta^2dv_0,
\endaligned
\end{equation}
where $\omega=u^{\frac{1+\beta}{2}}$,   
  $A_\beta=\frac{|1+\beta|^2}{6|\beta|}$, $C_s$ and $K_s$ are as in (\ref{SI}). 
\end{lemma}
\begin{proof} 
For simplicity, we will hence forth abbreviate  $\int dv_0$ as $\int$.   
Applying the Sobolev inequality (\ref{SI}) to   
the function $f=\omega\eta$,  we get 
\[
\aligned
\left(\int \eta^4\omega^4 \right)^{{1}/{2}} &\leq C_s\int|\nabla(\omega\eta)|^2+K_s\int\omega^2\eta^2  \\
 &\leq 2C_s\int|\nabla\omega|^2\eta^2+2C_s\int\omega^2|\nabla\eta|^2+K_s\int\omega^2\eta^2.
\endaligned
\]

Next we multiply both sides of
\[6\Delta u+S_gu^3=S_0u \]
by $\eta^2u^{\beta}$ and integrate, to get  
\[
6\beta\int_M\eta^2u^{\beta-1}|\nabla u|^2 +12\int_M \nabla u\cdot\nabla \eta \eta u^{\beta} +S_0\int_M\eta^2u^{\beta+1} =\int_MS_gu^2\eta^2u^{\beta+1}.\]
We can control the second term via
\[
\left|2\int_M\nabla u\nabla \eta \eta u^{\beta}\right|
\leq
\frac{1}{t}\int_M|\nabla\eta|^2u^{\beta+1} + t\int_M\eta^2|\nabla u|^2u^{\beta-1}. 
\]
For any $t$ with $0<t<|\beta|$,  
\[\aligned
-\frac{1}{t}\int|\nabla\eta|^2u^{\beta+1}-t\int\eta^2|\nabla u|^2u^{\beta+1} & \\
 \leq2\int (\nabla u & \cdot\nabla \eta)  \eta u^{\beta} 
 \\ & 
 \leq\frac{1}{t}\int|\nabla\eta|^2u^{\beta+1}+t\int\eta^2|\nabla u|^2u^{\beta-1}.
 \endaligned\]
It follows that for $\beta<0$ one has
\begin{equation}\label{beta-1}
6(|\beta|-t)\int\eta^2u^{\beta-1}|\nabla u|^2\leq\frac{6}{t}\int|\nabla\eta|^2u^{\beta+1}+|S_0|\int \eta^2u^{\beta+1}-\int S_gu^2\eta^2u^{\beta+1},
\end{equation}
while for $\beta>0$, one has
\begin{equation}
6(|\beta|-t)\int\eta^2u^{\beta-1}|\nabla u|^2\leq\frac{6}{t}\int|\nabla\eta|^2u^{\beta+1}+|S_0|\int \eta^2u^{\beta+1}+\int S_gu^2\eta^2u^{\beta+1}. 
\end{equation}
Take $t=\frac{|\beta|}{2}$ we get for $\beta<0$, 
\[\frac{12|\beta|}{|1+\beta|^2}\int|\nabla\omega|^2\eta^2\leq\frac{12}{|\beta|}\int|\nabla\eta|^2\omega^2+|S_0|\int\omega^2\eta^2-\int S_gu^2\omega^2\eta^2 \]
and for $\beta>0$, 
\[\frac{12|\beta|}{|1+\beta|^2}\int|\nabla\omega|^2\eta^2\leq\frac{12}{|\beta|}\int|\nabla\eta|^2\omega^2+|S_0|\int\omega^2\eta^2+\int S_gu^2\omega^2\eta^2. \]

So if $\beta<0$, we get
\[
\begin{split}
\left(\int \eta^4\omega^4\right)^{{1}/{2}} \leq &\frac{2C_s|1+\beta|^2}{12|\beta|} \left(\frac{12}{|\beta|}\int|\nabla\eta|^2\omega^2+ |S_0| \int\omega^2\eta^2 -  \int S_gu^2\omega^2\eta^2 \right) \\
& + 2C_s\int w^2|\nabla\eta|^2+K_s\int\omega^2\eta^2\\
 \leq  & 2C_s\left(\frac{|1+\beta|^2}{|\beta|^2}+1\right)\int|\nabla\eta|^2\omega^2+\left[2C_s\frac{|1+\beta|^2|S_0|}{12|\beta|^2}+K_s\right]\int \omega^2\eta^2\\
& + 2C_sC_g^{\frac{1}{2}}\frac{|1+\beta|^2}{12|\beta|^2}\left(\int(\omega\eta)^4\right)^{{1}/{2}}-2C_s\frac{|1+\beta|^2}{12|\beta|^2}\frac{\int S_gu^4}{\int u^4}\int u^2\omega^2\eta^2.
\end{split}
\]
Similarly, when $\beta>0$, we have
\[
\begin{split}
\left(\int \eta^4\omega^4\right)^{{1}/{2}} &\leq 2C_s\left(\frac{|1+\beta|^2}{|\beta|^2}+1\right)\int|\nabla\eta|^2\omega^2+\left[2C_s\frac{|1+\beta|^2|S_0|}{12|\beta|^2}+K_s\right]\int \omega^2\eta^2\\
&+ 2C_s\frac{|1+\beta|^2}{12|\beta|^2}C_g^{{1}/{2}}\left(\int(\omega\eta)^4\right)^{\frac{1}{2}}+2C_s\frac{|1+\beta|^2}{12|\beta|^2}\frac{\int S_gu^4}{\int u^4}\int u^2\omega^2\eta^2.
\end{split}
\]
This completes the proof.
\end{proof}

The following lemma is an analogue of lemma 1 in section 3 of \cite{CY1}:
\begin{lemma}\label{lem4.2} 
Suppose $(\ref{star})$ fails for any subsequence of a sequence of positive functions $\{u_j\}$ that satisfy 
\[\int_Mu_j^4dv_0=C_0,\] 
then $u_j\rightarrow0$ in $L^p$ for any $1 \le p<4$.
\end{lemma}
\begin{proof} 
For each $r>0$ we set 
\[\Omega_{r,j}\triangleq\{x \in M: u_j(x)\geq r\}.\]
We argue by contradiction. Suppose the lemma fails, i.e. there exists some $p<4$ and $\delta_0>0$ such that 
\[\int u_j^pdv_0\geq\delta_0\] 
for some subsequence of $u_j$, which we still denote by $u_j$ for simplicity. Then for each $r>0$ we have
\[\aligned 
\delta_0 \le \int u_j^pdv_0& = \int_{\Omega_{r,j}}u_j^pdv_0+\int_{M\setminus\Omega_{r,j}}u_j^pdv_0
\\ & \leq\left(\int u_j^4\right)^{{p}/{4}} \, \mathrm{Vol}(\Omega_{r,j})^{{(4-p)}/{4}}+ r^p  \, \mathrm{Vol}(\Omega_{r,j}^c). 
\endaligned\]
Choose  $r_0$  small so that 
\[r_0^p \, \mathrm{Vol}(M, g_0)<\frac{\delta_0}{2},\] 
then we get
\[\frac{\delta_0}{2}\leq C_0^{{p}/{4}} \, \mathrm{Vol}(\Omega_{r_0,j})^{{(4-p)}/{4}}.\]
Thus  
\[ \mathrm{Vol}(\Omega_{r_0,j})\geq  \left(\frac{\delta_0}{2C_0^{p/4}} \right)^{{4}/{(4-p)}}=:l_0\] 
for each $u_j$, which contradicts with our assumption that the condition  (\ref{star})  fails for the sequence $\{u_j\}$.
\end{proof}

The proof of the following proposition is sililar to the proof of the proposition B in section 3 of \cite{CY1}. The main differences are that we use lemma \ref{Lem4.1} and \ref{lem4.2} for 4-dimensional manifolds, while they use their formula (9b) and lemma 1 in their paper for 3-manifolds. For completeness, we will give the detail of the proof in the appendix. 
\begin{proposition}[\cite{CY1}, proposition B] \label{newpropB} 
Suppose $\{u_j\}$ is a sequence of positive functions defined on $(M^4,g_0)$ such that $g_j=u_j^2g_0$ satisfy the following conditions 
\begin{enumerate}
\item $a_0(g_j)=\alpha_0$,
\item $a_1(g_j)\leq\alpha_1$,
\item $\int S_{g_j}^2u_j^4dv_0\leq\alpha_2$,
\item $0<\Lambda\leq\lambda_1(g_j)$, 
\item The  condition  (\ref{star})  fails for any subsequence of $\{u_j\}$.
\end{enumerate} 
Then there exists some subsequence of $\{u_j\}$ 
whose mass concentrates at some point $x_0\in M$.  
\end{proposition} 
 
The next lemma is served as a replacement of Lemma 2 in \cite{CY1}. 
\begin{lemma}\label{newlemma2} 
Let $u$ be any positive function on $M$. 
Then for each point $x\in M$, there exists some neighborhood $\Omega(x)$ such that for every point $y\in\Omega(x)$ and geodesic ball $B(y,\rho)\subset\Omega(x)$ we have
\begin{equation}\label{krho3}
\int_{B(y,\rho)}|\nabla\log u|dv_0 \leq k\rho^3,
\end{equation}
where $k$ is a constant depending only on  $c_2=\int S^2_{g}u^4dv_0$, where $g=u^2g_0$.  In particular, there exists  a constant $p_0>0$ that depends only on $c_2$, such that
\begin{equation}\label{crho8}
\int_{B(y,\rho)}u^{p_0}dv_0 \int_{B(y,\rho)}u^{-p_0}dv_0\leq c\rho^8
\end{equation}
\end{lemma}
 
\begin{proof}
The proof is similar to that of \cite{CY1}, so we only describe the difference here.  
By choosing a cut-off function $\eta$ satisfying $|\nabla \eta| \le \frac {2}{\rho}$ on $B_{2\rho}$, taking $\beta=-1$ and $t=\frac{1}{2}$ in (\ref{beta-1}),  we get the following replacement of (21) in \cite{CY1}:
\[
3\int_{B_{\rho}}u^{-2}|\nabla u|^2\leq 12\int_{B_{2\rho}}\frac{4}{\rho^2}+|S_0|\int_{B_{2\rho}}1 +\int_{B_{2\rho}}|S_g|u^2.
\]
Since
\[
\int_{B_{\rho}}|S_g|u^2\leq \left(\int S_g^2u^4\right)^{{1}/{2}}\mathrm{Vol}(B_{2\rho})^{{1}/{2}}, 
\] 
we immediately see 
\[\int \frac{|\nabla u|^2}{u^2}\leq k_1\rho^2\]
for some $k_1=k_1(c_2)$. So  
\[\int_{B_{\rho}}|\nabla\log u|dv_0=\int_{B_{\rho}}\bigg|\frac{\nabla u}{u}\bigg|dv_0\leq 
\left(\int \frac{|\nabla u|^2}{u^2} \right)^{{1}/{2}} \mathrm{Vol}(B_{\rho})^{\frac{1}{2}}\leq k \rho^3.\]

The proof of (\ref{crho8}) goes the same as in \cite{CY1}. Namely, we only need to apply the Jonh-Nirenberg inequality (\cite{JN}, see also \cite{GT}) to the function $\log u$.
\end{proof} 

Finally by applying the Nash-Moser iteration as in \cite{CY1}, with their lemma 2 replace by our lemma \ref{newlemma2}, one can prove the following proposition. (The proof will be included in the appendix.) 
\begin{proposition}[\cite{CY1}, Proposition C and Remark]\label{newPropC}
Suppose $\{u_j\}$ is a sequence of functions as in Proposition \ref{newpropB} with $x_0$ its concentration point. Then for each fixed $r$ that is small enough and each $p \ge 2$, there exists some integer $j(r, p)$ and  some universal constant $C=C(p, p_0)$ so that
\begin{equation}
\int_{B(x_0,r)-B(x_0,\frac{r}{2})}u_j^pdv_0\leq C\int_{B(x_0,2r)-B(x_0,r)}u_j^pdv_0
\end{equation}
for all $j\geq j(r, p)$.  
\end{proposition}

 We end this section by the following proposition that serves as a replacement of proposition D in \cite{CY1}:
\begin{proposition}\label{ReplacePD}
Let $\{u_j\}$ be a sequence of positive smooth functions that satisfies the assumptions of Proposition \ref{newpropB}, and so that the conformal sequence $g_j=u_j^2g_0$ satisfies (\ref{R4B}). 
Then there exist constants $C_j>0$ with $C_j\to\infty$ so that the sequence
\[
v_j=C_j u_j
\]
converges uniformly on compact subset of $M\setminus\{x_0\}$ to the Green's function of the conformal Laplacian $L  = -6 \Delta_0  +S_0 $.
\end{proposition}
\begin{proof}
For simplicity we denote $B_r=B(x_0,r)$ and $B_r^c=M\setminus B_r$.
In what follows we will fix a small ball $B_r$ and choose a constant $C_j$ so that 
\[C_j^4\int_{B_r^c}u_j^4dv_0=1.\]
According to Proposition \ref{newpropB}, it is clear that $C_j\to\infty$.
 
By (\ref{R4B}), one can find constant $D$ such that 
\[\int_M|S_{g_j}|^4dv_{g_j}\leq D.\] 
 Notice that if we denote $\tilde g_j=v_j^2 g_0$, then 
\[ 
\int_{B_r^c}(\frac{Lv_j}{v_j^2})^4dv_0
=\int_{B_r^c}S^4_{\tilde g_j}v_j^4dv_0
=\frac{1}{C_j^4}\int_{B_r^c}S^4_{g_j}u_j^4dv_0.  
\]
So  as $j\to\infty$,
\[
\int_{B_r^c}(Lv_j)^{4/3}dv_0 
\leq \left(\int_{B_r^c}(\frac{Lv_j}{v_j^2})^4dv_0 \right)^{1/3} \left(\int_{B_r^c}v_j^4dv_0 \right)^{{2}/{3}}
\leq \left( \frac{D}{C_j^4} \right)^{1/3}\to 0. 
\]
It follows that  
$\{v_j\}$ has a subsequence converges strongly in $W^{2, 4/3}$ to a solution $\omega$ of the equation 
\[L\omega\equiv 0\] 
on $B_r^c$. We need to verify that $\omega$ is strictly positive. Since $S_0>0$, according to the minimum principle  for elliptic operators it is enough to prove  $\omega\not\equiv 0 $ on $B_r^c$.

Assume on the contrary that $\omega\equiv0$ on $B_r^c$. Then 
\[\lim\limits_{j\to\infty}\int_{B_r^c}v_j^2=\int_{B_r^c}\omega^2=0.\]
This implies 
\[\lim\limits_{j\to\infty}\int_{B_{r/2}^c}v_j^2 = 0, \]
since by Proposition \ref{newPropC},  
\[\int_{B_{{r}/{2}}^c}v_j^2=C_j^2(\int_{B_r\setminus B_{{r}/{2}}}u_j^2+\int_{B_r^c}u_j^2)\leq C C_j^2\int_{B_r^c}u_j^2=C \int_{B_r^c}v_j^2.\]

Also if we apply Proposition \ref{newPropC} with $p=4$, we get
\[\int_{B_{{r}/{2}}^c}v_j^4=C_j^4(\int_{B_r\setminus B_{{r}/{2}}}u_j^4+\int_{B_r^c}u_j^4)\leq \widetilde C C_j^4\int_{B_r^c}u_j^4=\widetilde C\int_{B_r^c}v_j^4=\widetilde C,\]
where $\widetilde C$ is a constant that depends only on $p_0$. 

On the other hand, by choosing a cut-off function $\eta$ so that
\[|\nabla\eta|\leq\frac{c}{r}, \quad \mbox{\ and\ }\quad \eta\equiv1  \mbox{\ on\ } B_r^c, \quad   \eta\equiv0  \mbox{\ on\ } B_{\frac{r}{2}},\] 
and applying  Lemma \ref{Lem4.1} with $\beta=1$ and $u=v_j$ we get
\begin{equation*}
\begin{split}
\left(\int_M \eta^4 v_j^4\right)^{{1}/{2}} 
&\leq \frac{C}{r^2}\int_{B_{{r}/{2}}^c}v_j^2+\frac{2C_sC_{\tilde g_j}^{{1}/{2}}}{3}\left(\int_M v_j^4\eta^4\right)^{{1}/{2}}+\frac{2C_s}{3}\frac{\int_M S_{\tilde g_j}v_j^4}{\int_Mv_j^4} \int_M v_j^4\eta^2
\\ & 
=\frac{C}{r^2}\int_{B_{{r}/{2}}^c}v_j^2+\frac{2C_sC_{g_j}^{{1}/{2}}}{3}\left(\int_M v_j^4\eta^4\right)^{{1}/{2}}+\frac{2C_s}{3}\frac{\int_M S_{g_j}u_j^4}{\int_Mu_j^4}\frac{1}{C_j^2}\int_{B_{r/2}^c} v_j^4.
\end{split}
\end{equation*}
Dividing both sides by $\left(\int_M \eta^4 v_j^4\right)^{1/2} \ge 1$ and letting $j \to \infty$ we get
\[1 \le \frac{2C_sC_{g_j}^{{1}/{2}}}{3},\]
which contradicts to (\ref{SCB}).

The rest of the proof goes exactly like \cite{CY1}: One can apply the standard diagonal trick to construct a sequence of functions $v_j=c_ju_j$, such that on $M \setminus \{x_0\}$, $v_j$ converges to a positive solution $\omega$ of $L\omega=0$. Then we apply the isolated singularity theorem of Gilbarg-Serrin \cite{GS} to conclude that $\omega\sim d(x,x_0)^{-2}$ which is the Green function of the conformal Laplacian.
\end{proof}

\section{Proof of theorem \ref{mainthm2}}

Finally we will prove theorem \ref{mainthm2}, which will also complete our proof of theorem \ref{mainthm}.

\begin{proof}[Proof of theorem \ref{mainthm2}]

We first prove that
\[g_{\omega} := \omega^2g_0 \] 
defines a flat metric on $M \setminus \{x_0\}$.  
In fact, by (\ref{R4B}) one can find a constant $C$ so that 
\[\int_M  |W(g_j)|_{g_j}^4 u_j^4 dv_{0}\le C.\]
So as $j \to \infty$ we have 
\[\int_M  |W(\tilde g_j) |_{\tilde g_j}^4 v_j^4 dv_{0} = \frac 1{C_j^4} \int_M  |W(g_j) |_{g_j}^4 u_j^4 dv_{0}  \to 0,\]
 where $\tilde g_j=v_j^2 g_0$ as we used in the proof of proposition \ref{ReplacePD}.   
Now let $K$ be any compact subset of $M\setminus\{x_0\}$. Then  
\[Lv_j\to L {\omega}\equiv0\] 
in $L^{4/3}$, hence $v_j \to \omega $ in $W^{2,4/3}$. By Fatou lemma,
\[
\int_K |W(g_\omega)|^4_{g_{\omega}}\omega^4 dv_{0} \leq\varliminf_{j\to\infty}\int_K | W(\tilde g_j)|_{\tilde g_j}^4 v_j^4dv_0=0.
\]
Hence the Weyl tensor $W(g_\omega)=0$ on $M \setminus \{x_0\}$. 
Apply the same argument to $Ric$, we also get $Ric(g_\omega)=0$ on $M \setminus \{x_0\}$. So $g_\omega$ is a flat metric on $M \setminus \{x_0\}$.

Next we prove that the Euler characteristic $\chi(M) > 0$. Suppose on the contrary that $\chi(M) \le 0$. Then we take $r>r'$ small and denote $A=B_r(x_0)$, $B=M\setminus \overline{B_{r'}(x_0)}$ such that $B$ can be retractable to $M\setminus \{x_0\}$. Then $A\cap B$ is homotopic  to $S^3$, while $A\cup B=M$.
So we have the following   Mayer-Vietoris sequence
\[\aligned
\to &   H_4(S^3)\to H_4(A)\oplus H_4(B)\to H_4(M)
\\ & \to H_3(S^3) \to H_3(A)\oplus H_3(B)  \to H_3(M)
\\ & {\color{white}\to} \to H_2(S^3)\to H_2(A)\oplus H_2(B)  \to H_2(M)
\\ &  {\color{white}\to\to} \to H_1(S^3) \to H_1(A)\oplus H_1(B)  \to H_1(M)
\\ &  {\color{white}\to\to\to} \to H_0(S^3)\to H_0(A)\oplus H_0(B)\to H_0(M)\to 0.
\endaligned\]
We can rewrite it as 
\[\aligned
\to & 0\to H_4(B)\xrightarrow{\phi_1}H_4(M) 
\\ & 
\xrightarrow{\phi_2}\mathbb{Z}\xrightarrow{\phi_3}H_3(B)\xrightarrow{\phi_4}H_3(M)
\\ &  {\color{white}\to} 
\to 0 \to H_2(B)\to H_2(M)
\\ &  {\color{white}\to\to} 
\to 0\to H_1(B)\to H_1(M)
\\ &  {\color{white}\to\to\to} \to  \mathbb{Z}\to\mathbb{Z}^2\to\mathbb{Z}\to 0,
\endaligned\]
from which we get isomorphisms 
\[H_1(M)\cong H_1(B), \quad H_2(B)\cong H_2(M)\]
and the relation
\[
rank~H_4(B) -rank H_4(M)+1-rank H_3(M))+rank H_3(B)=0.
\] 
So our assumption $\chi(M) \le 0$ implies $\chi(B)\le -1$ i.e. $\chi(M\setminus\{x_0\})\le -1$. This contradicts with corollary 3.3.5 in \cite{Wolf}.  

So we must have $\chi(M)>0$. According to \cite{CGY}, $M$ is diffeomorphic to either $S^4$ or $\mathbb{RP}^4$. We now show that $M$ cannot be diffeomorphic to $\mathbb{RP}^4$.  Our argument is close to the proof of corollary 3.3.5 in \cite{Wolf}. In fact, suppose  $M \simeq \mathbb{RP}^4$, then we have 
\begin{equation}\label{H1Mx0}
\pi_1(M \setminus \{x_0\}) \cong \pi_1(M) \cong\mathbb{Z}_2.
\end{equation}  
But by theorem 3.3.3 and theorem 3.3.1 in \cite{Wolf}, $M \setminus \{x_0\}$ admits a deformation retraction onto a compact totally geodesic submanifold $N$, and $N$ admits an $r$-fold covering by a torus $T$ for some $r>0$. If $\dim T>0$, then $\pi_1(N)$ contains a free abelian subgroup, which  contradicts to (\ref{H1Mx0}). If $\dim T=0$, then $M \setminus \{x_0\}$ contracts to a point, which contradicts to (\ref{H1Mx0}) again.

So we finally arrive at the conclusion that $M$ is diffeomorphic to $S^4$. Hence we have 
\[\chi(M \setminus \{x_0\}) =1.\]
We can apply  corollary 3.3.5 in \cite{Wolf} to conclude that  $(M\setminus\{x_0\},w^2g_0)$ is isometric to the flat $\mathbb{R}^4$.
Hence by Liouville's theorem,  $(M, g)$ is conformally equivalent to the round sphere $(S^4, g_0)$. This completes the proof.
\end{proof}

\section*{Appendix: The proofs of proposition 5.3 and 5.5}
\renewcommand\thesection{A}

For completeness, we will include the detailed proofs of proposition 5.3 and proposition 5.5 here. 
\begin{proof}[Proof of Proposition 5.3] $ $ 

As in [CY1], we will prove the proposition in two steps.
\begin{itemize}
\item Step I: The set of points where the mass of some subsequence of $\{u_j\}$ accumulates is nonempty:
\[A= \left\{x\in M:\lim_{r\rightarrow0}\varlimsup_{j\rightarrow\infty}\int_{B(x,r)}u_j^4\neq0\right\} \ne \emptyset.\]
\item 
Step I\!I:  The set $A$ above consists of exactly one point $x_0$.
\end{itemize}

\noindent\underline{Proof of Step I}. Suppose for each $x\in M$, we have
\[m_x:=\lim\limits_{r\to0}\varlimsup\limits_{j\to\infty}\int_{B(x,r)}u_j^4=0.\]
Then for any $x\in M$ fixed and any $\varepsilon>0$, one can find a subsequence $\{u_j\}$ so that  as $j\to\infty$ and $r$ sufficiently small,
\[\int_{B(x,r)}u_j^4<\varepsilon.\] 
We fix $r$ small and choose a cut-off function $\eta$ such that
\[0\leq \eta\leq 1,\quad \eta\equiv1 \text{\ on\ }B(x,\frac{r}{2}), \quad  \eta\equiv 0 \text{\ off\ }B(x,r),\quad \text{and\ }|\nabla\eta|\leq\frac{c}{r}.\]
Applying lemma 5.1 to $\beta=1,\omega=u_j$ and the above $\eta$, and notice 
\[\int_M u_j^2 \omega^2 \eta^2 \le \left(\int_{\mathrm{supp}\eta} \omega^4\eta^4  \right)^{1/2} \left(\int_{\mathrm{supp}\eta}u_j^4\right)^{1/2},\]
we obtain, for some constant $C$, 
\[\aligned
(\int\eta^4\omega^4)^{{1}/{2}}\leq & \frac{2C_sC_g^{{1}/{2}}}{3} \left(\int\eta^4\omega^4\right)^{{1}/{2}}+\frac{C}{r^2}\int_{B(x,r)}u_j^2dv_0\\ & \quad +\frac{2C_s}{3}\left|\frac{\int_M S_gu_j^4dv_0}{\int_Mu_j^4dv_0}\right| \left(\int_{B(x,r)} \omega^4\eta^4  \right)^{1/2} \left(\int_{B(x,r)}u_j^4\right)^{1/2}
\\ \le & \left(\frac 1{12}+\varepsilon^{\frac 12} \frac {2C_s}3 \left|\frac{a_1}{a_0}\right|\right)\left(\int\eta^4\omega^4\right)^{{1}/{2}}+\frac{C}{r^2}\int_{B(x,r)}u_j^2dv_0,
\endaligned\]
where we used $C_sC_g^{{1}/{2}}\leq\frac{1}{8}$. So for $\varepsilon$ sufficiently small and $j$ large we have
\[
\frac{1}{2}\left(\int_{B(x,{r}/{2})}u_j^4dv_0\right)^{{1}/{2}} \le \frac{1}{2}\left(\int_{B(x, r)}\eta^4\omega^4\right)^{{1}/{2}}\leq\frac{C}{r^2}\int_{B(x,r)}u_j^2dv_0.
\]
Now we cover $M$ by finitely many such balls $B(x_1,\frac{r_1}{2}), \cdots, B(x_N,\frac{r_N}{2})$. Then
\[a_0=\int_M u_j^4dv_0  \leq \sum_{k=1}^N\int_{B(x_k,{r_k}/{2})}u_j^4dv_0 \leq 4\sum_{k=1}^N\left(\frac{C}{r_k^2}\int_{B(x_k,r_k)}u_j^2dv_0\right)^2 \to 0,
\]
where we used lemma 5.2, which is a contradiction.

\vspace{.3cm}
\noindent \underline{Proof of Step I\!I}. Assume we have at least two points $x_1, x_2 \in A$. By passing to a subsequence of $\{u_j\}$ we may assume \[\lim\limits_{r\to0}\lim\limits_{j\to\infty}\int_{B(x_k,r)}u_j^4=m_k, \quad  k=1, 2.\]
As in [CY1] we let $\rho=\mathrm{dist}(x_1,x_2)$ and set
\[\mu_1=\limsup\limits_{j\to\infty}\int_{B(x_1,1)\setminus B(x_1,{1}/{2})}u_j^4dv_0,\]
which, by passing to a subsequene, becomes a limit. We then inductively choose subsequences of subsequences so that 
\[\mu_l=\lim_{j\to\infty}\int_{B(x_1,2^{-l+1})\setminus B(x_1,2^{-l})}u_j^4dv_0.\] 
Since $\sum_{j=1}^\infty \mu_l \leq a_0$, we can find $l_0$ so that for $l\geq l_0$ we have
\[
\mu_l\leq\frac{(a_0)^{{1}/{2}}}{4}\left(\frac{1}{m_1}+\frac{1}{m_2}\right)^{-{1}/{2}}.
\]
We do the same argument near $x_2$ and choose a common $l_0$. Then we pick $\rho_0\leq \min\{\rho/2,2^{-l_0}\}$ small so that for all $r\leq2\rho_0$ and $j$ sufficiently large,
\[\left|\int_{B(x_k,r)}u_j^4dv_0-m_k\right|\leq\varepsilon, \quad  k=1, 2.\]  
Then we choose  $\phi$ to be a $C^{\infty}$-function on $M$ with
\[
\phi=\begin{cases}
\frac{1}{m_1}, &\text{on $B(x_1,\rho_0)$},\\
-\frac{1}{m_2}, &\text{on $B(x_2,\rho_0)$},\\
0, &\text{$(B(x_1,2\rho_0)\cup B(x_2,2\rho_0))^c$}
\end{cases}
\]
and we extend $\phi$ ``linearly" in the rest of $M$, then
$$\int u_j^4\phi^2\geq \frac{m_1-\varepsilon}{m_1^2}+\frac{m_2-\varepsilon}{m_2^2}\geq\frac{1}{2}\left(\frac{1}{m_1}+\frac{1}{m_2}\right).$$
and
\[
\begin{split}
\left| \int \phi u_j^4 \right| &\leq \left|\int_{(B(x_1,\rho_0)\cup B(x_2,\rho_0))}\phi u_j^4\right|+ \left|\int_{\cup_{k=1}^2(B(x_k,2\rho_0)\setminus B(x_k,\rho_0))}|\phi|u_j^4\right|\\
&\leq \left(\frac{m_1+\varepsilon}{m_1}-\frac{m_2-\varepsilon}{m_2}\right)+\frac{a_0^{{1}/{2}}}{4}\left(\frac{1}{m_1}+\frac{1}{m_2}\right)^{{1}/{2}}\\
&\leq \left(\frac{1}{m_1}+\frac{1}{m_2}\right)\varepsilon+\frac{a_0^{{1}/{2}}}{4}\left(\frac{1}{m_1}+\frac{1}{m_2}\right)^{{1}/{2}}
\end{split}
\]
and 
\[\int |\nabla\phi|^2u_j^2\leq C\left(\frac{1}{\rho_0^2m_2^2}+\frac{1}{\rho_0^2m_1^2}\right)\int_Mu_j^2.\]
According to the Rayleigh-Ritz inequality, for any smooth function $\varphi$ we have
\[
\int \varphi^2 dv \le \left(\int dv\right)^{-1}\left(\int \varphi dv\right)^2 + \frac 1{\lambda_1} \int |\nabla \varphi|^2 dv.
\]
Apply this to the metric $g=u_j^2g_0$, we get
\[
\lambda_1  \leq \frac{a_0\int|\nabla\phi|^2u_j^2}{a_0\int u_j^4\phi^2-(\int u_j^4\phi)^2}
\leq \frac{C a_0  \left(\frac{1}{\rho_0^2m_2^2}+\frac{1}{\rho_0^2m_1^2}\right)\int_Mu_j^2}
{ \frac{a_0}{4}(\frac{1}{m_1}+\frac{1}{m_2})}
\to 0,
\]
where we used lemma 5.2 again in the last step.
This contradicts with the isospectrality.

\end{proof}

\begin{proof}[Proof of Proposition 5.5] $ $ 

Given any $p \ge 2$, we choose $\bar p$ to be the smallest number of the form $3 \times 2^l$ that is greater than or equal to $p$, where $l$ is a non-negative integer. 
Let $p_0$ be a constant so that lemma 5.4 holds. Note that by H\"older's inequality, $p_0$ can be taken to be any sufficiently small number, and we will take $p_0$ to be a number such that $0<p_0<\frac 14$ and such that there exists $m \in \mathbb N$ with $2^{m+1}p_0=\bar p$. 

As in [CY1], we denote $B_r=B(x_0, r)$ and $B_{r_1, r_2}=B_{r_1}\setminus \overline{B_{r_2}}$. We will fix $r$ small enough such that  $B_{4r}$ is contained in a normal coordinate patch at $x_0$. Now for $k=0,1,2,...,m$ we define $\delta_k$ and $\beta_k$ by 
\[\delta_k=\frac{r}{2^{k+3}}, \qquad 1+\beta_k=2^kp_0.\] 
Then for each $k \le m$ we have $|1+\beta_k| \le \bar p/2$. Moreover, the minimum of $|\beta_k|$ is attained at $|\beta_{m-1}|=1/4$. So we get 
$\frac{|1+\beta_k|^2}{|\beta_k|}\leq \bar p^2$ and $\frac{|1+\beta_k|^2}{|\beta_k|^2}\leq 4\bar p^2$ for all $k$.
Next we let $\rho_m=r, \sigma_m=\frac{r}{2}$, and define $\rho_k, \sigma_k (1 \le k \le m-1)$ by 
\[\rho_{k-1}=\rho_k+\delta_k, \quad \sigma_{k-1}=\sigma_k-\delta_k.\]
Note that for each $k$ we have $\sigma_k-\delta_k> r/4$, so $\mathrm{supp}\,\eta_k \subset B_{\sigma_k-\delta_k}^c \subset B_{{r}/{4}}^c$. Moreover, by definition $\rho_0+\delta_0<3r$. As a consequence, for each $k$ the triples of numbers $(\rho, \sigma, \delta)=(\rho_k, \sigma_k, \delta_k)$ satisfies $2\delta<\sigma< \rho < 3r$ and $\delta<r$. We choose a smooth cut-off function $\eta=\eta_k$  so that 
\[\eta\equiv1 \text{ on } B_{\rho,\,\sigma}, \quad \mathrm{supp}\,\eta \subset B_{\rho+\delta,\,\sigma-\delta} , \quad |\nabla\eta|\lesssim \frac{1}{\delta}  \text{\ on its support.}\] 

Now for $\beta=\beta_k$ and  $u=u_j$ with $j$, we let 
\[A_{\beta, u, \eta}=1-C_s\frac{|1+\beta|^2}{6|\beta|}\left({\int_M S_g^4u^4}\right)^{1/4} \left(\int_{\mathrm{supp}\,\eta}u^4 \right)^{1/4}.\]
Then for $j$ large enough we have $A_{\beta, u, \eta}>1/2$, since 
\begin{enumerate}
\item by our choice of $\beta=\beta_k$, we have $ {|1+\beta_k|^2}/{|\beta_k|} \leq \bar p^2$, 
\item according to Proposition 3.5, the integral $\int_M S_g^4u^4$ is bounded, 
\item we have $\mathrm{supp}\,\eta \subset B^c_{r/4}$, so for $j>j(r,p)$ large enough  the   integral $\int_{\mathrm{supp}\eta} u^4$ is sufficiently  small.
\end{enumerate}   
Applying the last step of the proof of lemma 5.1 to $\eta$ and $\omega=u^{\frac{1+\beta}{2}}$, we get
\[
\begin{split}
\left(\int \eta^4\omega^4\right)^{{1}/{2}} \leq & 2C_s\left(\frac{|1+\beta|^2}{|\beta|^2}\int|\nabla\eta|^2\omega^2+\frac{|1+\beta|^2|S_0|}{12|\beta|}\int\omega^2\eta^2\right. \\
& \left.+\frac{|1+\beta|^2}{12|\beta|}\bigg|\int S_gu^2\omega^2\eta^2\bigg|\right) + 2C_s\int w^2|\nabla\eta|^2+K_s\int\omega^2\eta^2\\
\leq & 2C_s\left(\frac{|1+\beta|^2}{|\beta|^2}+1\right)\int|\nabla\eta|^2\omega^2+\left[2C_s\frac{|1+\beta|^2|S_0|}{12|\beta|}+K_s\right]\int \omega^2\eta^2\\
&+ C_s\frac{|1+\beta|^2}{6|\beta|}\left({\int_M S_g^4u^4}\right)^{1/4} \left(\int_{\mathrm{supp}\,\eta}u^4 \right)^{1/4}\left(\int \eta^4\omega^4\right)^{{1}/{2}},
\end{split}
\]
which implies
\begin{equation}\label{5.5.1}
A_{\beta, u, \eta}\left(\int_{B_{\rho, \sigma}}\omega^4\right)^{{1}/{2}}\leq \frac{B_{\beta}}{\delta^2}\int_{B_{\rho+\delta, \sigma-\delta}}\omega^2,
\end{equation}
where 
\[B_{\beta}=C\left(\frac{|1+\beta|^2}{|\beta|^2}+\frac{|1+\beta|^2}{|\beta|}|S_0|+1\right),
\]
and $C$ is a universal constant that depends only on $C_s, K_s$. We denote 
\[\Phi(u,p,\Omega)=\left(\int_{\Omega}u^p\right)^{{1}/{p}}
.\] 
Since $A_{\beta, u, \eta}>1/2$ and $\frac{|1+\beta_k|^2}{|\beta_k|^2}\leq 16p^2$, there is a universal constant $C'(p)$ such that 
$B_\beta / A_{\beta, u, \eta}<C'(p)$. 
So the formula (\ref{5.5.1}) reads
\begin{equation}\label{5.5.2}\Phi(u,2(1+\beta),B_{\rho, \sigma})\leq (C'/\delta^2)^{1/|1+\beta|} \Phi(u,1+\beta,B_{\rho+\delta, \sigma-\delta}) \text{ if $1+\beta>0$}\end{equation} 
and
\begin{equation}\label{5.5.3}\Phi(u,1+\beta,B_{\rho+\delta, \sigma-\delta})\leq  (C'/\delta^2)^{1/|1+\beta|} \Phi(u,2(1+\beta),B_{\rho, \sigma}) \text{ if $1+\beta<0$}.\end{equation}

Now we iteratively apply (\ref{5.5.2}) to get
\[\aligned
\Phi(u_j,\bar p,B_{r, {r}/{2}})& \leq\prod_{k=0}^m (C'/\delta_k^2)^{1/|1+\beta_k|} \Phi(u_j,p_0,B_{\rho_0+\delta_0, \sigma_0-\delta_0})\\
&\leq\prod_{k=0}^m (C'/\delta_k^2)^{1/|1+\beta_k|} \Phi(u_j,p_0,B_{3r, {r}/{4}}).
\endaligned\]
Let $c=\sum_{k=0}^m \frac{2k+6}{2^k}$. Then 
\[\prod_{k=0}^m \left(\frac{C'}{\delta_k^2}\right)^{\frac 1{|1+\beta_k|}} \le (C')^{\sum \frac 1{2^kp_0}}2^{{\sum \frac{(2k+6)}{2^{k}p_0}}}r^{{-2 \sum\frac 1{2^{k}p_0}}} 
\le (C')^{\frac 2{p_0}}2^{\frac c{p_0}}r^{-4(\frac{1}{p_0}-\frac{1}{\bar p})}.\]
So we get
\begin{equation}\label{5.5.4}
\Phi(u_j,\bar p,B_{r, {r}/{2}})\leq  (C')^{\frac 2{p_0}}2^{\frac c{p_0}}r^{-4(\frac{1}{p_0}-\frac{1}{\bar p})} \Phi(u_j,p_0,B_{3r, {r}/{4}}).
\end{equation}

Similarly we let $\tilde \beta_k$ be such that $1+\widetilde\beta_k=-2^kp_0$, let $\tilde \delta_k =\frac r{2^{k+3}}$ as before, let $\tilde \rho_m=2r, \tilde \sigma_m=r$ and let $\tilde \rho_{k-1}= \tilde \rho_k+\tilde \delta_k, \tilde \sigma_{k-1}=\tilde \sigma_k-\tilde \delta_k$ for $0 \le k \le m$. One can check that all the inequalities we need among these quantities are satisfied. We then iteratively use (\ref{5.5.3}) to get 
\begin{equation}\label{5.5.5}\aligned
\Phi(u_j,-p_0,B_{3r, {r}/{4}}) &\le \Phi(u_j,1+\tilde \beta_0,B_{\tilde \rho_0+\tilde \delta_0, \tilde \sigma_0-\tilde \delta_0}) \\
&\leq\prod_{k=0}^m (C'/\tilde \delta_k^2)^{1/|1+\tilde \beta_k|}  \Phi(u_j,-\bar p,B_{2r, r}) \\
& \le   (C')^{\frac 2{p_0}}2^{\frac c{p_0}}r^{-4(\frac{1}{p_0}-\frac{1}{\bar p})} \Phi(u_j,-\bar p,B_{2r, r}).
\endaligned
\end{equation}
 
By definition, 
\[ \Phi(u_j,p_0,B_{3r, {r}/{4}})^{p_0}=\int_{B_{3r, {r}/{4}}}u_j^{p_0}\leq\int_{B_{3r}}u_j^{p_0}\]
and 
\[\Phi(u_j,-p_0,B_{3r,{r}/{4}})^{-p_0}=\int_{B_{3r, {r}/{4}}}u_j^{-p_0}\leq\int_{B_{3r}}u_j^{-p_0}.\]
Apply lemma 5.4 to $u=u_j$ with $B_{\rho}=B_{3r}$, we get
\[
\Phi(u_j,p_0,B_{3r, {r}/{4}})^{p_0}\Phi(u_j,-p_0,B_{3r, {r}/{4}})^{-p_0}\leq\int_{B_{3r}}u_j^{p_0}\int_{B_{3r}}u_j^{-p_0}\leq C(p_0)r^8,
\]
or in other words,
\begin{equation}\label{5.5.6}
\Phi(u_j,p_0,B_{3r, {r}/{4}})   \leq C'(p_0)r^{8/p_0}\Phi(u_j,-p_0,B_{3r, {r}/{4}}).
\end{equation}
So we get
\[\aligned
\Phi(u_j,{p},B_{r,{r}/{2}}) &\leq C_1\Phi(u_j,{\bar p},B_{r,{r}/{2}})r^{\frac4p-\frac 4{\bar p}} \qquad \qquad \text{(by H\"older)} \\
& \le C_2 r^{\frac {4}{\bar p}-\frac 4{p_0}} \Phi(u_j, p_0, B_{3r, r/4})r^{\frac4p-\frac 4{\bar p}}  \qquad \qquad \text{(by (\ref{5.5.4}))}\\
& \le C_3 r^{\frac {4}{\bar p}-\frac 4{p_0}} r^{\frac 8{p_0}}\Phi(u_j,-p_0,B_{3r, {r}/{4}})r^{\frac4p-\frac 4{\bar p}}  \qquad \qquad \text{(by (\ref{5.5.6}))}\\
& \le C_4 r^{\frac {4}{\bar p}-\frac 4{p_0}} r^{\frac 8{p_0}}r^{\frac {4}{\bar p}-\frac 4{p_0}} \Phi(u_j,-\bar p,B_{2r, r})r^{\frac4p-\frac 4{\bar p}}  \qquad \quad \text{(by (\ref{5.5.5}))}\\
& = C_4 r^{\frac 4p + \frac 4{\bar p}} \Phi(u_j,-\bar p,B_{2r, r}) \\
& \leq  C_5 r^{\frac 8p} \Phi(u_j,-p,B_{2r, r}) \qquad \qquad \text{(by H\"older)}\\
& \leq  C_6 r^{\frac 8p}r^{-\frac 8p}  \Phi(u_j,p,B_{2r,r}) \qquad\qquad\text{(by Cauchy-Schwartz)}\\
& = C_6 \Phi(u_j,p,B_{2r,r}).
\endaligned\]
where $C_6$ is a constant that depends only on $p_0$. 
This completes the proof. 
\end{proof}

\end{document}